\documentclass[11pt,reqno]{amsart}
\usepackage{amsmath,amssymb}

\renewcommand{\epsilon}{\varepsilon}
\renewcommand{\phi}{\varphi}
\newcommand{\cov}{\operatorname{ cov}}

\newcommand{\su}{\subseteq}
\newcommand{\rest}{\restriction}
\renewcommand{\a}{\alpha}
\renewcommand{\b}{\beta}

\renewcommand{\d}{\delta}
\renewcommand{\l}{\lambda}
\renewcommand{\k}{\kappa}

\newcommand{\z}{\zeta}

\newcommand{\om}{\omega}

\newcommand{\lng}{\langle}
\newcommand{\rng}{\rangle}
\newcommand{\ov}{\overline}
\newcommand{\sm}{\setminus}

\newcommand{\dom}{{\operatorname {dom}\,}}
\newcommand{\cf}{{\operatorname {cf}}}

\newcommand{\otp}{{\operatorname {otp}}}
\newcommand{\ran}{{\operatorname  {ran}}}

\newcommand{\tcf}{\operatorname {tcf}}

\newcommand{\imply}{\Rightarrow}

\newtheorem{theorem}{Theorem}[section]
\newtheorem{corollary}[theorem]{Corollary}

\newtheorem{definition}[theorem]{Definition}
\newtheorem{claim}[theorem]{Claim}

\newtheorem*{abstheorem}{Theorem}

\author{Menachem Kojman}

\address{Department of Mathematics\\
Ben-Gurion University of the Negev\\
P.O.B. 653 \\
Be'er Sheva\\
84105 Israel}

\thanks{Research on this paper was partially supported by an Israeli Science Foundation grant number 1365/14}
\email{kojman@math.bgu.ac.il}



\begin{document}
\title{On the arithmetic of  Density}
\maketitle

\begin{abstract}The $\k$-density of a cardinal $\mu\ge\k$ is the least cardinality of a dense collection of $\k$-subsets of $\mu$ and is denoted by $\mathcal D(\mu,\k)$. The \emph{Singular Density Hypothesis} (\textsf{SDH}) for a singular cardinal $\mu$ of cofinality $\cf\mu=\k$ is  the equation $\mathcal D(\mu,\k)=\mu^+$. The \emph{Generalized Density Hypothesis} (\textsf{GDH}) for $\mu$ and $\l$ such that  $\l\le\mu$ is:
\[
\mathcal D(\mu,\l)=
\begin{cases} 
 \mu & \text{ if } \cf\mu\not=\cf\l \cr
 \mu^+ & \text{ if } \cf\mu=\cf\l.
 \end{cases}
 \]

Density  is shown to satisfy Silver's theorem. The most important case is:
\begin{abstheorem}[Theorem \ref{densesilver}]
If $\k=\cf\k<\theta=\cf\mu<\mu$ and the set of cardinals $\l<\mu$ of cofinality $\k$ that satisfy the \textsf{SDH} is stationary in $\mu$ then the \textsf{SDH} holds at $\mu$. 
\end{abstheorem}

A more general version is given in Theorem \ref{general}.

A corollary of Theorem \ref{densesilver} is: 

\begin{abstheorem}[Theorem \ref{gdh}] If the Singular Density Hypothesis  holds for  all sufficiently large singular cardinals of some fixed  cofinality $\k$, then for all cardinals $\l$ with  $\cf\l \ge k$, for all sufficiently large   $\mu$, the  \textsf{GDH}   holds.
\end{abstheorem}
\end{abstract}

\section{Introduction}
 \emph{Eventual regularity} is a recurring theme in cardinal arithmetic since the discovery of pcf theory. Arithmetic rules that do not necessarily hold for all cardinals, can  sometimes be seen to  hold in appropriate end-segments of the cardinals.  

The most famous precursor of modern cardinal arithmetic is \emph{Silver's theorem} \cite{silver}, which says that if one of the arithmetic equations  (1) the \emph{Singular Cardinal Hypothesis} (\textsf{SCH}); or (2) the  \emph{Generalized Continuum Hypothesis} (\textsf{GCH}), holds sufficiently often below a singular of uncountable cofinality, then it   holds  at the singular itself.  

Silver's theorem came as a surprise in 1973,  shortly after  Solovay and  Easton employed Forcing, that was discovered by Cohen in 1963, to prove that  no non-trivial bound  on the power of a regular cardinal could be deduced from  information about the  powers of smaller cardinals. At the time,   all set theorists believed that no such implications existed and that further development of Forcing would clear  the missing singular case soon (see \cite{singular} for the history of the subject and for a survey of other precursors of pcf theory, e.g. in topology).


\medskip

The present note concerns the eventual regularity of  the  cardinal arithmetical function \emph{density}.
The density function $\mathcal D(\mu,\k)$ is defined for cardinals $\k\le \mu$ as the least cardinality of a collection $\mathcal D\su [\mu]^\k$ which is \emph{dense} in $\lng [\mu]^\k, \su\rng$. 

A detailed definition and basic properties of density appear  in  Section 2 below.  
Let us point out  now, though, 
 one crucial  difference between $\mathcal D(\mu,\k)$ and the exponentiation $\mu^\k$: the function $\mathcal D(\mu,\k)$ is \emph{not} monotone increasing in the second variable. For example,  if $\mu$ is a strong limit cardinal of cofinality $\omega$ then $\mathcal D(\mu,\aleph_0)=\mu^+ > \mathcal D(\mu,\aleph_1)=\mu$. 

Recently, asymptotic results in infinite graph theory and in the combinatorics of families of sets \cite{graphs, splitting} --- some of which were proved earlier with the \textsf{GCH} or with forms of the \textsf{SCH} \cite{eh-graphs,eh-families,hjss,kom-close,kom-comp} ---  were proved in \textsf{ZFC} by making use of  an eventual regularity property of density:  that  density satisfies a version of Shelah's  \textsf{RGCH} theorem. 
 See also \cite{sh:1052} on the question whether the use of  \textsf{RGCH}  in \cite{graphs} is necessary. 

\subsection{The results}
Three theorem about the eventual behaviour of density are proved below. Theorems \ref{densesilver} is a density versions of the most popular case of Silver's theorem and Theorem \ref{general} is a density version of the general  Silver theorem.  They deal with the way the behaviour of density at singular cardinals of cofinality $\k$ below a singular $\mu$ of cofinality $\theta>\kappa$ bounds the $\theta$-density at $\mu$. 

The proofs of \ref{densesilver} and of \ref{general} follow  in their  outline  two elementary proofs 
by Baumgartner and Prikry:   \cite{BaumPrik}, for the central case, and \cite{BaumPrikDisc}, for the general theorem.  
The following   modifications were required. First, one has to use almost disjoint families of sets instead of general families. The reason is that in the pressing down argument with the density function is not injective in general, but is so with the additional condition of almost disjointness.  Second, a use of a pcf scale in the proof of Theorem \ref{densesilver}  replaces an indirect argument in \cite{BaumPrik}. This is not strictly necessary, but makes the proof clearer. Finally, the density of stationary subsets with inclusion replaces the stronger hypothesis about cardinal arithmetic in the general case. 

An elementary proof of Silver's theorem   was  discovered in 1973  also  by Jensen, independently of \cite{BaumPrikDisc}, but was only circulated and not published   (see the introduction to \cite{BaumPrikDisc} and  \cite{singular}).

\medskip
Theorem  \ref{gdh} states that if the  \textsf{SDH} holds eventually at some fixed cofinality $\k$ then    the \textsf{GDH} holds for all sufficiently large cardinals $\mu$ and $\l\le \mu$ such that  $\cf\l\ge\k$. 
The proof is by induction, and employs Theorem \ref{densesilver}   in  the critical cases.

 \subsection{Notation and prerequisites}
The notation used here is standard in set theory. In particular, the word \emph{cardinal}, if not explicitly stated otherwise, is to be understood as ``infinite cardinal". The variables $\k,\theta,\mu,\l$ stand for infinite cardinals and $\a,\b,\gamma,\d,i,j$ denote ordinals. By $\cf\mu$ the \emph{cofinality} of $\mu$ is denoted. For $\k<\mu$ the symbol $[\mu]^\k$ denotes the set of all subsets of $\mu$ whose cardinality is $\k$. 

We assume familiarity with the basics of stationary sets and the non-stationary ideal and acquaintance with Fodor's pressing down theorem. This material is available in every standard set theory textbook.

\subsection{Potential use in topology}
We conclude the introduction with the following  illustration of  the potential applicability of density to topology. 

Suppose $G=\lng V, E\rng$ is an arbitrarily large graph (one can assume that it is a proper class with no harm) and that $G$ does not conatain large bipartite graphs, say, for some cardinal $\l$ there is no copy of the complete bipartite graph $K_{\l,\l}$ in $G$. 

For every cardinal $\mu$, let us define a topology on $V$ by letting $U\su V$ be open if for all $v\in U$ it holds that $|G[v]\sm U|<\mu$ ($G[v]$ is the \emph{set of neighbours} of $v$ in $G$). Equivalently, $D\su V$ is closed if  
every vertex $v\in V$ which is connected by edges to $\mu$ vertices from $D$ belongs to $D$.

What can be said about the cardinalities of closed sets in this topology? Using the arithmetic properties of the density function, 
it was proved in \cite{graphs} that  if $\mu\ge \beth_\om(\l)$, the closure of every set of size $\theta\ge \mu$ has 
size $\theta$.

\section{Definition and Basic properties of density}

\begin{definition}
\begin{enumerate} 
\item If  $\lng P,\le\rng$ is a partially ordered set and $A, B\su P$ then $A$ is \emph{dense in $B$} if 
$(\forall y\in B)(\exists x\in A)(x\le y)$. We say that $A\su P$ is \emph{dense} if $A$ is dense in $P$. 
\item If  $\lng P,\le\rng$ is a partially ordered set and $A,B\su P$ then $A$ is  an \emph{antichain with respect to  $B$}  
if for all distinct $x,y\in A$ there is no $z\in B$ such that $z\le x \wedge z\le y$. We say that $A\su B$ is an antichain if $A$ is an antichain with respect to $P$. 
\end{enumerate}
\end{definition}

\begin{definition} Suppose $\theta\le \mu$ are cardinals. 
\begin{enumerate}
\item The \emph{$\theta$-density of $\mu$}, denoted by  $\mathcal D(\mu,\theta)$, is  the least cardinality of a 
set  $\mathcal D\su [\l]^\k$ which is \emph{dense} in $\lng [\mu]^\theta,\su\rng$. 
\item Let $\overline{[\mu]^\theta}=\{X\in [\mu]^\theta: \forall \a(\a<\theta\imply |X\cap \l|<\theta)\}$.  
\item Let $\underline{ [\mu]^\k}=\bigcup\{[\a]^\theta: \a<\mu\}$ (the set of all  members of $[\mu]^\theta$ which are bounded in $\mu$).  
\item Let $\overline{\mathcal D(\l,\theta)}$ be the least cardinality of a set $\mathcal D\su \overline{[\mu]^\theta}$ 
which is dense in $\overline{[\mu]^\theta}$, and let us call it \emph{the upper $\theta$-density of $\mu$},   and let $\underline{ \mathcal D(\mu,\theta)}$ be the least cardinality 
of $\mathcal D\su \underline{ [\mu]^\theta}$ which is dese in $\underline{[\mu]\theta}$, and let us call it \emph{the lower $\theta$-density of $\mu$}. 
\end{enumerate}
\end{definition}

\begin{claim} Suppose $\theta\le\mu$. Then $\mathcal D(\mu,\theta) = \overline{\mathcal D(\mu,\theta)}+
 \underline{ \mathcal D(\mu,\theta)}$. 
\end{claim}

\begin{proof}
Given any $X\in [\mu]^\theta$, either there is some $\l<\mu$ such that 
$Y:=X\cap \l$ is of cardinality $\theta$ or else $X\in \overline {[\l]^\theta}$. Thus,
 $\overline{[\mu]^\theta}\cup \underline{ [\mu]^\theta}$ is dense in $[\mu]^\theta$ and therefore
  $\mathcal D(\mu,\theta) \le  \overline{\mathcal D(\l,\theta)}+
 \underline{ \mathcal D(\mu,\theta)}$ by taking the union of  dense subsets of 
 $\overline{[\mu]^\theta}$ and of $\underline{ [\mu]^\theta}$ of minimal cardinalities. 

Conversely, given a dense $\mathcal D\su [\mu]^\theta$ of minimal cardinality let 
$D_0=\mathcal D\cap \underline{[\mu]^\theta}$
 and let $\mathcal D_1=\mathcal D\cap \overline{[\mu]^\theta}$. Clearly, $\mathcal D_0$ is dense in 
 $\underline{ [\mu]^\theta}$. To see that $\mathcal D_1$ is dense in $\overline{[\mu]^\theta}$ let 
 $Y\in \overline{[\mu]^\theta}$ be arbitrary. Since $\mathcal D$ is dense, there is some 
 $X\in \mathcal D$ such that $X\su Y$. For all $\l<\mu$ it holds that $X\cap \l\su Y\cap \l$, so 
 $X\in \overline{[\mu]^\theta}$, and now $X\in \mathcal D_1$. 
\end{proof}

\noindent\textbf{Remark}: If $X\in \overline{[\mu]^\theta}$ then $\otp X=\theta$ and $X$ is cofinal in $\mu$, so consequently
 $\cf\theta=\cf\mu$. Thus, if $\cf\theta\not=\cf\mu$ it holds that  $\overline{[\mu]^\theta}=\emptyset$ and that $\mathcal D(\mu,\theta)=\underline{\mathcal D(\mu,\theta)}$.

\begin{claim} \label{sum}Suppose $\theta=\cf\mu<\mu$. Then:
\begin{enumerate}
\item Every maximal antichain in $\lng \overline{ [\mu]^\theta},\su \rng$ has cardinality $\ge \mu^+$. 
\item  $\overline{\mathcal D(\mu,\theta)}=|\mathcal A|+\mathcal D(\theta,\theta)$ whenever 
$\mathcal A\su \overline{[\mu]^\theta}$ is a maximal antichain in $\overline{[\mu]^\theta}$. 
\end{enumerate}
\end{claim}

\begin{proof}The first item is proved by standard diagonalization. 

To prove the second let $\mathcal A\su \overline{[\mu]^\theta}$ be a maximal antichain in  $\overline{[\mu]^\theta}$. Since the intersection of two distinct members from $\overline{[\mu]^\theta}$ belongs to $\overline{[\mu]^\theta}$ if an only if it belongs to $[\mu]^\theta$, $\mathcal A$ is an antichain in $[\mu]^\theta$ as well. 

For every $X\in \mathcal A$ fix a dense $\mathcal D_X$ in $\lng[X]^\theta,\su\rng$ of cardinality $\mathcal D(\theta,\theta)$ and let $\mathcal D=\bigcup\{ \mathcal D_X: X\in \mathcal A\}$. The cardinality of $\mathcal D$ is $|\mathcal A|+\mathcal D(\theta,\theta)$ and as every $Z\in \mathcal D_X$ for $X\in \mathcal A$ belongs to $\overline{[\mu]^\theta}$, we have that $\mathcal D\su \overline{[\mu]^\theta}$. Given any $Y\in \overline{[\mu]^\theta}$, there exists some $X\in \mathcal A$ such that $Y\cap X\in [X]^\theta$ and therefore there is some $Z\in \mathcal D_X$ such that $Z\su Y$. This establishes that $\mathcal D$ is dense in $\overline{[\mu]^\theta}$. 

Conversely, let $\mathcal D\su \overline{[\mu]^\theta}$ be dense in $\overline{[\mu]^\theta}$ and let $\mathcal A\su \overline{[\mu]^\theta}$ be an antichain in $\overline{[\mu]^\theta}$. Let $f:\mathcal A\to D$ be such that $f(X)\su X$ for all $X\in \mathcal A$. As $\mathcal A$ is an antichain, $f$ is injective and hence $|\mathcal A|\le |\mathcal D|$. If $X\in \overline{[\mu]^\theta}$ then $\mathcal D\cap [X]^\theta$ is dense in $[X]^\theta$ and hence $\mathcal D(\theta,\theta)\le |\mathcal D|$. 
\end{proof} 

\begin{corollary}\label{anti} If $\theta=\cf\mu<\mu$ and $\mathcal D(\theta,\theta)\le \mu^+$  then every maximal antichain in $\overline{[\mu]^\theta}$ has
 cardinality $\overline{\mathcal D(\mu,\theta)}$. 
\end{corollary}



%
%
%
%
%
%

We phrase  now  the first Theorem. It is a version of Silver's theorem for the density function. 

\begin{theorem}\label{densesilver} Suppose $\k=\cf\k<\theta=\cf\mu<\mu$. 
If $\underline{\mathcal D(\mu,\theta)}\le \mu^+$  and the set $\{\l<\mu: \cf\l=\k \wedge \overline{\mathcal D(\l,\k)}=\l^+\}$ 
is a stationary subset  of $\mu$,
then $\mathcal D(\mu,\theta)=\mu^+$. 
\end{theorem}

 \begin{proof}Since $\underline{\mathcal D(\mu,\theta)} \le \mu^+\le\overline{\mathcal D(\mu,\theta)}$ 
 and $\mathcal D(\mu,\theta)= \underline{\mathcal D(\mu,\theta)} + \overline{\mathcal D(\mu,\theta)}$, it holds that $\mathcal D(\mu,\theta)= \overline{\mathcal D(\mu,\theta)}$. 
 By Corollary \ref{anti}, $\mathcal D(\mu,\theta)$ is equal to the cardinality of any maximal antichain in $\lng \overline {[\mu]^\theta},\su\rng$. So it suffices to prove that $|\mathcal A|\le \mu^+$ for any  antichain
 $\mathcal A\su \overline{[\mu]^\theta}$.  Fix such an antichain $\mathcal A$.

 Let $\lng \k_i:i<\theta\rng$ be a strictly increasing and continuous sequence 
 of cardinals converging to $\mu$ with $\theta<\k_0$. Let $S=\{i<\theta: \cf i =\k \wedge \overline{\mathcal D(\k_i,\k)}=\k_i^+\}$. By the assumptions, $S$ is a stationary subset of $\theta$.

  By a basic pcf theorem \cite{CA}, $\tcf {\lng \prod_{i\in S}\k_i^+,<_{NS\rest S}\rng}=\mu^+$, so we can fix \emph{ a pcf scale}
$\ov f=\lng f_\a:\a<\mu^+\rng\su \prod_{i<\theta} \k_i^+$, that is, a sequence which is $<_{NS\rest S}$-increasing ---
 $\a<\b<\mu^+ \imply \{i\in S: f_\a(i)\ge f_\b(i)\}$ is non-stationary --- and $<_{NS\rest S}$-cofinal --- for every $f\in \prod_{i\in S}\k_i^+$
 there is some $\a<\mu^+$ such that  $\{i\in S: f(i)>f_\a(i)\}$ is non-stationary. We shall only use  the cofinality of the scale.

Now fix, for every  $i\in S$ a
  dense set $\mathcal D_i$  in $\lng \overline{[\k_i]^\k},\su\rng$ and an enumeration
 $\mathcal D_i=\{Z^i_j:j<\k_i^+\}$, and  also a
  dense set $\mathcal D'_i\su [\theta\times \k_i]^\theta$ with $|\mathcal D'|\le \mu^+$. 
 
  Given $X\in \mathcal A$ let $C_X=\{i<\theta: \forall j (j< i \imply X\cap \k_i \not\su k_j)\}$, which is clearly a club of
   $\theta$. Let $S^1_X=C_X\cap S$. As an intersection of a club with a stationary subset, $S^1_X$ is stationary in $\theta$. 
   
   For each $i\in S^1_X$ let $f_X(i)=j<\k_i^+$ be such that
    $Z^i_j\in [\k_i]^{\k}$ is a subset of $X\cap \k_i$ which is cofinal in $\k_i$ and of order-type $\k$. 
    Such $j$ exists because $X\cap \k_i$ is cofinal, $\cf\k_i=\k$ and $\mathcal D_i$ is dense in $\lng \overline{[\k_i]^{\k}},\su\rng$. 
    
    As 
 $\ov f$ remains a scale when $NS\rest S$ is extended to $NS\rest S^1_X$, there is some $\a(X)<\mu^+$
  such that $f_X<_{NS\rest S^1_X} f_{\a(X)}$.

  \begin{claim}For every  $\a<\mu^+$, at most $\mu^+$ many $X\in \mathcal A$ satisfy that $\a=\a(X)$. 
  \end{claim}
  
  \begin{proof}
  Let $\a<\mu^+$ be fixed and for every $i\in S^\theta_\k$ let us fix an injection $g_i:f_\a(i)\to \k_i$. 
  Suppose $X\in \mathcal A$ satisfies that $\a=\a(X)$, so $f_X<_{NS\rest S^1_X} f_\a$. 
  
  By shrinking $S^1_X$ we may assume that $f_X(i)<f_\a(i)$ for all $i\in S^1_X$. 
  For each $i\in S^1_X$ let $r(i)=\min\{j<i: g(f_X(i))<\k_j\}$. Since $\k_i$ is limit, $r$ is well-defined and is a regressive function on $S^1_X$. 
  
  By Fodor's lemma, there is some 
  stationary $S^2_X\su S^1_X$ and  some fixed $j(X)<\theta$ such that $r(i)=j(X)$ for all $i\in S^2_X$. 
  Let $h_X(i):=g_i(f_X(i))\in k_{j(X)}$ for all $i\in S^2_X$. Now the function $h_X:S^2_X\to \k_{j(X)}$ (which is a set of ordered pairs) 
  is a subset of
   $\theta\times \k_{j(X)}$.   Let $Z(X)\in \mathcal D'_{j(X)}$ be chosen such that $Z(X)\su h_X$ (so $Z(X)$ is a partial function from $\theta$ to $\k_{j(X)}$). 
  
  Suppose $X,Y\in \mathcal A$ are distinct and suppose that $j(X)=j(Y)$. If $Z=Z(X)=Z(Y)$ then $\dom Z$ is unbounded in $\theta$ and for every  $i\in \dom Z$ the set $Z^i_{f_X(i)}=Z^i_{f_Y(i)}$ is unbounded in $\k_i$ and contained in $X\cap Y$. Hence $|X\cap Y|=\theta$ --- a contradiction to $|Z\cap Y|<\theta$.   
  
 Thus, the mapping $X\mapsto\lng j(X), Z(X)\rng$ is injective 
  on the set of all $X\in \mathcal A$ such that $f_X<_{NS} f_\a$. As there are at most 
  $|\mathcal D'_{j(X)}|+\theta\le \mu^+$ such pairs, we are done. 
   \end{proof}
  The theorem follows  immediately from the claim. 
  \end{proof}

  \subsection{The general version}Throrem \ref{densesilver} above is formulated after the most popular version of Silver's theorem. Silver's original paper as well as \cite{BaumPrikDisc} included, however, a more general formulation, involving the $\gamma$-th successors of $\k_i$ and of $\mu$ for arbitrary ordinals $\gamma<\theta$. The case $\gamma=0$ in the general case is actually a theorem by Erd\H os, Hajnal and Milner from 1967 about almost disjoint families \cite{EHM} (for more on the history  see  \cite{aki}). 
  
%
%
%
%

\medskip
Let $\mathcal S^\theta_\k$, for $\k=\cf\k<\theta=\cf\theta$,  denote the family of all stationary subsets of $S^\theta_\k=\{\a<\theta:\cf\a=\k\}$.

\begin{theorem}\label{general} Suppose $\k=\cf\k<\theta=\cf\mu<\mu$ and that $\lng \k_i:i<\theta\rng$ is an 
increasing and continuous sequence of cardinals with limit $\mu$ and $\theta<\k_0$. 

Let $\mathcal A\su \lng \overline{[\mu]^\theta},\su\rng$ be an antichain and let $\gamma<\theta$ be an ordinal. 

Suppose that there exists  a sequence $\lng \mathcal D_i:i\in S^\theta_\k\rng$ such that 
\begin{enumerate}
\item $\mathcal D_i\su \overline{[\k_i]^\k}$ and $|\mathcal D_i|\le \k_i^{+\gamma}$ for all $i\in S_\k^\theta$;
\item for every $A\in \mathcal A$ the set $S_A:=\{i\in S_\k^\theta: (\exists X\in \mathcal D_i)(X\su A)\}$ is stationary.
\end{enumerate}
 Then
    $|\mathcal A|\le \mu^{+\gamma}+\underline{\mathcal D(\mu,\theta)}+\mathcal D(\mathcal S^\theta_\kappa,\su)$.
\end{theorem}

We remark that  if $2^\theta<\mu$ then  $\mathcal D(\mathcal S^\theta_\kappa,\su)$ can be removed from the conclusion,  giving  $|\mathcal A|\le \mu^{+\gamma}+\underline{\mathcal D(\mu,\theta)}$, and if $\mathcal D(\k_i,\theta)<\mu$ for all $i$ then also $\underline{\mathcal D(\mu,\theta)}$ can be removed. In the latter  case the theorem  has a meaningful content also in the case $\gamma=0$.

\begin{proof} 
Suppose $\k,\theta,\mu,\mathcal A,\gamma$ and $\lng \mathcal D_i:i\in S^\theta_\k\rng$ are as stated in the hypothesis of the theorem and fix in addition, for each $i\in S^\theta_\k$, an injection $t_i:\mathcal D_i\to \k_i^{+\gamma}$. 
Fix   also a
  dense set $\mathcal D'_i\su [\theta\times \k_i]^\theta$ of  cardinality $\mathcal D(\k_i,\theta)$  and an enumeration $\{S_\zeta:\zeta<\zeta(*)\}$  of a dense subset of $\lng\mathcal S^\theta_\k,\su\rng$ for  $\zeta(*)=\mathcal D(\mathcal S,\su)$. 
  
  To save on notation let us abbreviate the term  $\mu^{+\gamma}+\underline{\mathcal D(\mu,\theta)}+\mathcal D(\mathcal S^\theta_\kappa,\su)$ by   $\l(\gamma)$ for each $\gamma<\theta$.

For each $A\in \mathcal A$ let $g_A:S_A\to \k_i^{+\gamma}$ by letting $g_A(i):=t_i(X)$ be the least  of such that $X\su A$

%

The proof  proceeds now by induction on $\gamma<\theta$ to show that $|\mathcal A|\le \l(\gamma)$.

Assume $\gamma=0$. Then for each $A\in \mathcal A$ and $i\in S(A)$ it holds that $g_A(i)<\k_i$.  By Fodor's lemma there is some $j(A)<\theta$ and a stationary $S^1_A\su S_A$ so that $\ran(h_A\rest S^1_A)\su \k_{j(A)}$.  
Let $Y(A)\in \mathcal D'_{j(A)}$ such that $Y(A)\su g_A\rest S^1_A$. As in the previous proof, the mapping $A\mapsto \lng j(A),Y(A)\rng$ is injective.  The number of possibile pairs $\lng j(A),Y(A)\rng$ is at most $\mathcal D(\k_i,\theta)\times \mu$ so we have established 
$|\mathcal A| \le \mu +\underline{\mathcal D(\k_i,\theta)}\le \l(0)$. Observe that  $\mathcal D(\mathcal S^\theta_\k,\su)$ was not used in this case!

Now assume $\gamma=\beta+1$. 

\begin{claim}\label{first} For every $g\in \prod_{i\in S}\k_i^{+\gamma}$ there are at most $\l(\beta)$ members $A\in \mathcal A$ for which
there exists  some stationary $S'\su S_A$ such that $g_A(i)<g(i)$ for all $i\in S'$.  
\end{claim}

\begin{proof}
Let $g\in \prod_{i\in S}\k_i^{+\gamma}$ be given, and let $D^g_i:=\{X\in \mathcal D_i:t_i(X)<g(i)\}$. Thus 
the injection  $t_i\rest \mathcal D^g_i$ demonstrates that  $|\mathcal D^g_i|\le |g(i)|\le \k_i^{+\beta}$. 
Finally, let $\mathcal A_g=\{A\in \mathcal A:(\exists S'\in \mathcal S)(S'\su S_A\wedge (g_A\rest S')<g\}$.
Now $\mathcal A_g$, $\beta$ and $\lng \mathcal D^g_i:i\in S^\theta_\k\rng$ satisfy the hypothesis of the theorem and the conclusion follows by the induction hypothesis. 
\end{proof}

For $\zeta<\zeta(*)$,  let $\mathcal A_\zeta=\{A\in A:\zeta(A)=\zeta\}$. This correspondence partitions $\mathcal A$ to at most $\zeta(*)$ subfamilies. 

%
%
%
%
%
%

%
%
 
 \begin{claim}
  $|\mathcal A_\z|\le \l(\beta)$ for every $\zeta<\zeta(*)$.
 \end{claim}
 
 \begin{proof} Consider  the relation $\mathbin{R}$ on $\mathcal A$ given by:
\[A \mathbin{R} B \iff \{i\in S^1_A\cap S^1_B : g_A(i)<g_B(i)\} \text{ is stationary}.\]

Observe that if $A,B\in \mathcal A$ are distinct then $\{i\in S_A\cap S_B:g_A(i)=g_B(i)\}$ is bounded in $\theta$, in particular non-stationary. So if $\zeta(A)=\zeta(B)$ and $A\not=B$ we have 
\begin{equation}
A\mathbin R B \vee B \mathbin R A.\label{vee}
\end{equation} 

(Both disjuncts may hold simultaneously). 

Let $\zeta<\zeta(*)$ be given. 
By Claim \ref{first}, for every $A\in \mathcal A_\z$ there are no more than $\l(\beta)$ members $B$ of $\mathcal A_\z$ for which $B\mathbin{R} A$. Define inductively, as long as possible,  an injective sequence $\lng A_\xi:\xi <\xi(*)\rng$ such that $\neg(A_\xi\mathbin{R} A_\phi)$ for all $\phi<\xi$. If $\xi(*)> \l(B)$, then as $\neg(A_{\l(B)}\mathbin R A_\phi)$ for all $\phi<\l(\beta)$ it follows by (\ref{vee}) that $A_\phi \mathbin R A_{\l(\beta)}$ for all $\phi<\l(\beta)$, and this contradicts Claim \ref{first}. 

Necessarily, then, $\xi(*)\le \l(\beta)$. Thus every $A\in \mathcal A_\z$ satisfies $A\mathbin R A_\xi$ for some $\xi<\xi(*)\le \l(\b)$ and another use of Claim \ref{first} gives the required $|\mathcal A_\z|\le \l(\b)$. 
 \end{proof}
 Now the inequality $|\mathcal A|\le \l(\gamma)$ follows easily.

 \smallskip

  Suppose finally that $0<\gamma<\theta$ is limit.   Since $\gamma<\theta$ and the non-stationary
   ideal  is $\theta$-complete,  for every $A\in \mathcal A$ there is some $\beta(A)$  such that $g_A(i)<\k_i^{+\beta(A)}$ stationarily often, so $|\mathcal A|\le \l(\gamma)$ by the induction hypothesis. 
 
\end{proof}

\section{The eventual \textsf{GDH} follows from the eventual \textsf{SDH}}

Let us now define  the \emph{Singular Density Hypothesis} and the \emph{Generalized Density Hypothesis}
by modifying the well known \textsf{SCH} and \textsf{GCH}: 

\begin{definition}
\par\noindent
\begin{enumerate}
\item The \textsf{SDH} at a singular $\mu$ with $\cf\mu=\theta$ is the statement: 
  
  \[ \overline{\mathcal D(\mu,\theta)}=\mu^+. \tag{$\oplus$}\]

\item The \textsf{GDH} at a pair of cardinals $\l\le \mu$ is the statement: 
 \[
\mathcal D(\mu,\l)=
\begin{cases} \tag{$\otimes$}
 \mu & \text{ if } \cf\mu\not=\cf\l \cr
 \mu^+ & \text{ if } \cf\mu=\cf\l
 \end{cases}
 \]
\end{enumerate}
\end{definition}

Similarly to what    \textsf{SCH} and  \textsf{GCH}  say about cardinal exponentiation, the \textsf{SDH} says   that  the ``essential part\footnote{Compare this with the evolution of formulations of the \textsf{SCH} which is described in  \cite{singular}. The most modern and most informative one is $\cov (\mu,\theta)=\mu^+$. The role of $\cov (\mu,\theta)=\mu^+$ for exponentiation is played by upper density for the density function.}" of $\mathcal D(\mu,\theta)$ assumes the least possible value at a the  singular $\mu$ of cofinality $\theta$, 
 and the \textsf{GDH} says  that  the $\l$-density of $\mu$ assumes its minimal possible value. 
%
  
Let us define the \emph{Eventual} \textsf{GDH},  \textsf{EGDH}, for short, as the statement: there exists $\k$ such that for all $\l$ with $\cf\l\ge \k$ there is some $\mu_\l$ such that for all $\mu\ge \mu_\l$ the \textsf{GDH} holds at $\mu$ with $\l$.   
  
%
%

\begin{theorem}\label{gdh} If $\k$ is regular and the \textsf{SDH} holds for all $\mu$ with cofinality $\k$ in some end-segment of the cardinals, 
then the  \textsf{EGDH} holds:  for every  $\l$ with $\cf\l\ge \k$, for all $\mu$ in some end-segment of the cardinals the \textsf{GDH} holds at $\mu$ with $\l$. 
\end{theorem}

\begin{proof}  

%
%
%

 Suppose $\k$ is regular, $\mu_\k$ is a cardinal, and that  $\mathcal D(\mu,\k)=\mu^+$ for all singular $\mu\ge\mu_k$ with $\cf\mu=\k$. By replacing  $\mu_k$ with $(\mu_\k)^\k$, if necessary, we assume that $(\mu_\k)^\k=\mu_\k$. 
 
 We need to show that
 for every cardinal $\l$ with $\cf\l \ge \k$ there is an end-segment of the cardinals in which
\[
\mathcal D(\mu,\l)=
\begin{cases} \tag{$\otimes$}
 \mu & \text{ if } \cf\mu\not=\cf\l \cr
 \mu^+ & \text{ if } \cf\mu=\cf\l
 \end{cases}
 \]

By induction on $\l\ge\k$ we define a cardinal $\mu_\l$ and  for $\l$ with  $\cf\l\ge \k$ prove by induction on $\mu\ge\mu_\l$ that  $(\otimes)$ holds. 

The first case we consider is of a \emph{regular} $\l\ge \k$. Let  a regular $\l$ be given. If $\l=\k$ then $\mu_\l$ is already defined. If $\theta>\l$ let $\mu_\theta$ be chosen so that $\mu_\l\ge \mu_\k$ and $(\mu_\l)^\l=\mu_\l$. Now let us show by induction on $\mu\ge \mu_\l$ that $(\otimes)$ holds. If $\mu=\mu_\l$ then $\cf\mu>\l$ and $\mu\le\mathcal D(\mu,\l)\le \mu^\l=\mu$, so $(\otimes)$ indeed holds. 

Assume next that $\cf\mu\not=\l$. In this case for every $X\in [\mu]^\l$ there exists some $\a<\mu$ such that $X\cap \a\in [\a]^\l$. 
The induction hypothesis implies that $\mathcal D(\a,\l)\le |\a|^+\le \mu$, so $(\otimes)$ follows readily.

The remaining case is, then, $\cf\mu=\l$.  By the  induction hypothesis, $\underline{\mathcal D(\mu,\l)}=\l$. As $\mu_\k\le \mu_\l<\mu$, an end-segment of singulars $\mu'$ of cofinality $\k$ below $\mu$ satisfy $\mathcal D(\mu',\k)=\mu'^+$. By Theorem \ref{densesilver}, $\mathcal D(\mu,\l)=\mu^+$.

Assume now that $\l$ is singular. If $\cf\l<\k$ we are not really required to do anything, so let us define $\mu_\l$ as $0$. If $\cf\l=\theta\ge \k$ let $\mu_\l$ be chosen so that $\mu_\l>\mu_{\l'}$ for all $\l'<\l$ and $(\mu_\l)^\l=\mu_\l$. Now proceed to prove $(\otimes)$ by induction on $\mu\ge \mu_\l$. The cases $\mu=\mu_\l$  and $\cf\mu\not=\cf\l$ follow in the same way as for regular $\l$. 

We are left with the case $\cf\mu=\cf\l=\theta$ and $\l<\mu$. 
Fix an increasing sequence of regular cardinals  $\lng \l_i:i<\theta\rng$ that converges to $\l$, and such that $\theta <\l_0$. By the induction hypothesis
 on $\l$,  $(\otimes)$ holds for $\mu$ with each $\l_i$, so $\mathcal D(\mu,\l_i)=\mu$ for each $i$ and we can fix a dense $\mathcal D_i\su [\mu]^{\l_i}$ of cardinality $|\mathcal D_i|=\mu$.  Fix an  injection $f:\bigcup_{i<\theta}  \mathcal D_i\to \mu$.   
 
As $\theta<\l$ and $\mu_\theta<\mu$, the induction hypothesis (on $\l)$ implies that $\mathcal D(\mu,\theta)=\mu^+$. 
 Fix, then, a dense $\mathcal D_\theta\su [\mu]^\theta$ of cardinality $\mu^+$. Let $A_i=\ran (f\rest [\mu]^{\l_i})$ for $i<\theta$. Clearly, $|A_i|=\mu$ for each $i<\theta$ and as $[\mu]^{\l_i}\cap [\mu]^{\l_j}=\emptyset$ for $i<j<\theta$ and $f$ is injective, the $A_i$-s are pairwise disjoint.

Let $\mathcal D=\{\bigcup_{\a\in X}f^{-1}(\a): X\in \mathcal D_\theta \wedge |A_i\cap X|\le 1\wedge |\bigcup_{\a\in X}f^{-1}(\a)|=\l\}$. 

By the definition of $\mathcal D$ it is a subset of $[\mu]^\l$ and since $|\mathcal D_\theta|=\mu^+$, the cardinality 
of $\mathcal D$ does not exceed $\mu^+$. We prove next that $\mathcal D$ is dense in $[\mu]^\l$  (so apostriori $|\mathcal D|=\mu^+$) and with this finish the proof.
 
 Let $Y\in [\mu]^\l$ be arbitrary. For each $i<\l$ choose a set $Y_i\in [Y]^{\l_i}\cap \mathcal D_i$. This is possible
  since $\mathcal D_i$ is dense in $[\mu]^{\l_i}$. Let $Z=\{f(Y_i):i<\theta\}$. Clearly, $Z\in [\mu]^{\theta}$ and $|Z\cap A_i|=1$ for every $i<\theta$. 
  By the density of $\mathcal D_\theta$, there exists some $X\in \mathcal D_\theta$ such that $X\su Z$. Thus, $|X\cap A_i|\le 1$
   for each $i<\theta$. As $|X|=\theta$, for arbitrarily large $i<\theta$ it holds that $|X\cap A_i|=1$. It follows 
   that $\bigcup_{\a\in X}f^{-1}(\a)\su Y$ belongs to $\mathcal D$ and is contained in $Y$.

\end{proof}

\section{Concluding Remarks}
The  density function was not yet applied to topology, but it is reasonable to assume that applications will be found. 

If the \textsf{EGDH} holds, then for any two regular cardinals $\theta_1,\theta_2$ above $\k$,  for every sufficiently large $\mu$
\begin{equation}
\mu=\min\{\mathcal D(\mu,\theta_1),\mathcal D(\mu,\theta_2)\}.
\end{equation}

We do not know if the negation of the \textsf{EGDG} is consistent.  A harder consistency would be the negation of the following: 

\begin{itemize}
\item For every $\k$ there a finite set of cardinals $F$ above $\k$ and some $\mu_0$ such that for all $\mu\ge \mu_0$ 
\[\mu=\min\{\mathcal D(\mu,\theta):\theta \in F\}.\]
\end{itemize}

   Replacing ``finite" with ``countable" in this statement produces a \textsf{ZFC} theorem (see \cite{splitting}). 
   
%
%
%
%
%

\end{document}